\documentclass[12pt,reqno]{amsart}
\usepackage{amsmath,amsthm,amssymb}
\usepackage{amsfonts}
\usepackage{a4wide}
\allowdisplaybreaks
\numberwithin{equation}{section}

\newtheorem{theorem}{Theorem}[section]

\newtheorem{lemma}[theorem]{Lemma}

\theoremstyle{definition}

\newtheorem{remark}[theorem]{Remark}

\newcommand{\ds}{\displaystyle}

\newcommand{\Om}{\Omega}
\newcommand{\om}{\omega}
\newcommand{\la}{\lambda}

\newcommand{\va}{\varepsilon}
\newcommand{\be}{\beta}
\newcommand{\al}{\alpha}

\newcommand{\R}{{\mathbb R}}

\begin{document}
\title
[positive radial solutions for coupled Schr\"{o}dinger system] {Positive radial solutions for coupled Schr\"{o}dinger system with critical exponent in $\R^N\,(N\geq5)$}
\author{Yan-fang Peng \, and \,Hong-yu Ye$^*$
}

\address{Y. F. Peng, School of Mathematics and Statistics, Central China
Normal University, Wuhan, 430079, P. R. China }

\email{ pyfang2005@sina.com }

\address{$^*$ Corresponding author: H. Y. Ye, School of Mathematics and Statistics, Central China
Normal University, Wuhan, 430079, P. R. China}

\email{ yyeehongyu@163.com}

\begin{abstract}
We study the following coupled Schr\"{o}dinger system
\begin{equation*}
 \left\{
 \begin{array}{ll}
 \ds -\Delta u+u=u^{2^*-1}+\be u^{\frac{2^*}{2}-1}v^{\frac{2^*}{2}}+\la_1u^{\al-1},\,\,\, &x\in \R^N, \vspace{0.2cm}\\
 \ds -\Delta v+v=v^{2^*-1}+\be u^{\frac{2^*}{2}}v^{\frac{2^*}{2}-1}+\la_2v^{r-1},\,\,\, &x\in \R^N, \vspace{0.2cm}\\
 u,v\,>\,0\,, &x\in \R^N,
 \end{array}
 \right.
\end{equation*}
where $N\geq 5, \la_1,\la_2>0,\be\neq 0, 2<\al,r<2^*,2^*\triangleq \frac{2N}{N-2}.$ Note that the nonlinearity and the coupling terms are both critical. Using the Mountain Pass Theorem, Ekeland's variational principle and Nehari mainfold, we show that this critical system has a positive radial solution for positive $\be$ and some negative $\be$ respectively. \\

\end{abstract}

\maketitle
\section{Introduction}

In this paper, we consider the following coupled nonlinear Schr\"{o}dinger system
\begin{equation}\label{1.1}
 \left\{
 \begin{array}{ll}
 \ds -\Delta u+u=u^{2^*-1}+\be u^{\frac{2^*}{2}-1}v^{\frac{2^*}{2}}+\la_1u^{\al-1},\,\,\, &x\in \R^N, \vspace{0.2cm}\\
 \ds -\Delta v+v=v^{2^*-1}+\be u^{\frac{2^*}{2}}v^{\frac{2^*}{2}-1}+\la_2v^{r-1},\,\,\, &x\in \R^N, \vspace{0.2cm}\\
 u,v\,>\,0\,, &x\in \R^N,
 \end{array}
 \right.
\end{equation}
where $N\geq 5, \la_1,\la_2>0,\be\neq 0,2<\al,r<2^*,2^*\triangleq \frac{2N}{N-2}.$ We are interested in the existence of a nontrivial solution $(u,v)$ for \eqref{1.1}, that is to say that $u\not\equiv0$ and $v\not\equiv0$. We call a solution $(u,v)$ semi-trivial if $(u,v)$ is type of $(u,0)$ or $(0,v)$.

In recent years, there have been a lot of studies on the following coupled system of the time-dependent nonlinear Schr\"{o}dinger equations
\begin{equation}\label{1.2}
 \left\{
 \begin{array}{ll}
 \ds -i\frac{\partial}{\partial t}\Phi_1-\Delta \Phi_1=\mu_1|\Phi_1|^2\Phi_1+\beta|\Phi_2|^2\Phi_1,\,\,\, &x\in \Omega,~t>0, \vspace{0.2cm}\\
 -i\frac{\partial}{\partial t}\Phi_2-\Delta \Phi_2=\mu_2|\Phi_2|^2\Phi_2+\beta|\Phi_1|^2\Phi_2,\,\,\, &x\in \Omega,~t>0, \vspace{0.2cm}\\
 \Phi_j=\Phi_j(x,t)\in\mathbb{C},\,\, j=1,2,\vspace{0.2cm}\\
 \Phi_j(x,t)=0,\,\, j=1,2,\,\,\, &x\in \partial\Omega,t>0,
 \end{array}
 \right.
\end{equation}
where $\Omega=\R^N$ or $\Om\subset\R^N$ is a smooth bounded domain. $i$ is the imaginary unit, $\mu_1,\mu_2>0$ and a coupling constant $\beta\neq0$. When $N\leq 3$, system \eqref{1.2} appears in many physical problems, especially in nonlinear optics. Physically, the solution
$\Phi_j$ denotes the $j$th component of the beam in Kerr-like photorefractive media (see \cite{aa}). The positive constant $\mu_j$ is for self-focusing in the $j$th component of the beam. The coupling constant $\be$ is the interaction between the two components of the beam. The interaction is attractive if $\be>0$ while it is  repulsive if $\be<0$. System \eqref{1.2} also arises  in the Hartree-Fock theory for a binary mixture of Bose-Einstein condensates in two different hyperfine states, see more details in \cite{aa,egbb,m}.

To obtain solitary wave solutions of system \eqref{1.2}, we set $\Phi_1(x,t)=e^{i\la_1x}u(x),$ $\Phi_2(x,t)=e^{i\la_2x}v(x)$, then \eqref{1.2} turns to be the following elliptic system
\begin{equation}\label{1.3}
 \left\{
 \begin{array}{ll}
 \ds -\Delta u+\la_1u=\mu_1u^3+\be uv^2,\,\,\, &x\in \Om, \vspace{0.2cm}\\
 -\Delta v+\la_2v=\mu_2v^3+\be u^2v,\,\,\, &x\in \Om, \vspace{0.2cm}\\
 u,v\,=0\,, &x\in \partial\Om,
 \end{array}
 \right.
\end{equation}
 where $\Omega=\R^N$ or $\Om\subset\R^N$ is a smooth bounded domain, $\mu_1,\mu_2>0$ and $\beta\neq0$. When $N\leq3$, system \eqref{1.3} is a problem of subcritical growth. Problem \eqref{1.3} was first studied by Lin and Wei in \cite{lw} where they obtained a nontrivial solution when $\Omega=\R^3$ and $\beta>0$ is sufficiently small. After that, the existence and multiplicity of positive and sign-changing solutions  have been extensively studied, we can refer to \cite{ac1,ac2,bdw,bw1,bw2,cz1,cz2,dww,lw2,lw3,lw4,p,sk,ww1,ww2}. In particular, in \cite{cz1}, Chen and Zou studied problem \eqref{1.3} for $N=4$ and $\Om$ is a bounded domain of $\R^4$. In such case, the nonlinearity and the coupling terms are both of critical growth. By Ekeland's variational principle and the Mountain Pass Theorem, they showed that if $-\lambda_1(\Omega)<\lambda_1\leq\lambda_2<0$, then there exist $\beta_1\in (0,\min\{\mu_1,\mu_2\}),$ $\beta_2\geq\max\{\mu_1,\mu_2\}$ such that \eqref{1.3} has a positive least energy solution for $\beta\in (-\infty,0)\cap(0,\beta_1)\cap(\beta_2,+\infty)$, where $\la(\Om)$ is the first eigenvalue of $-\Delta$ with the Dirichlet boundary condition. Meanwhile, \eqref{1.3} does not have a nontrivial nonnegative solution if $\mu_2\leq\beta\leq\mu_1$ and $\mu_2<\mu_1$. In a similar way to \cite{cz1}, Chen and Zou in \cite{cz2} considered the following critically coupled nonlinear Schr\"{o}dinger equations:
 \begin{equation}\label{1.4}
 \left\{
 \begin{array}{ll}
 \ds -\Delta u+\la_1u=\mu_1u^{2^*-1}+\be u^{2^*/2-1}v^{2^*/2},\,\,\, &x\in \Om, \vspace{0.2cm}\\
 -\Delta v+\la_2v=\mu_2v^{2^*-1}+\be u^{2^*/2}v^{2^*/2-1},\,\,\, &x\in \Om, \vspace{0.2cm}\\
 u,v\,> 0,\,x\in \Om, u=v=0,\,x\in \partial\Om,
 \end{array}
 \right.
\end{equation}
where $\Om$ is a bounded domain in $\R^N(N\geq5)$. However, since $N\geq5$, different phenomenons may happen comparing to the case $N=4$ and it is much more complicated to handle. In \cite{cz2}, they proved that if $-\lambda_1(\Omega)<\lambda_1\leq\lambda_2<0$, \eqref{1.4} has a positive least energy solution for any $\beta\neq0$.
Furthermore, in \cite{sk}, by using the Mountain Pass Theorem, Kim showed that  problem \eqref{1.4} has a nontrivial solution in the following two cases: $\beta$ is sufficiently large or $|\beta|$ is small enough.

Problem \eqref{1.1} can be seen as a counterpart of the following problem:
\begin{equation}\label{1.5}
 -\Delta u+u=|u|^{2^*-2}u+f(u),\,\,x\in \R^N.
\end{equation}
In \cite{dyb}, under the following assumptions on $f(t)$:

$(f_1)$\,\,$\ds f\in C^2(\R^1), \lim_{t\rightarrow 0^+}\frac{f(t)}{t}=0, \lim_{t\rightarrow \infty}\frac{f(t)}{t^{2^*-1}}=0$ for $t\geq0$;

$(f_2)$\,\,There exists an $\varepsilon>0$ small enough such that
$$
tf'(t)\geq (1+\va)f(t)>0
$$
for $t>0$;

$(f_3)$\,\,$f(t)$ is odd,\\
Deng proved that when $N\geq4$, \eqref{1.5} has at least a positive least energy radial solution with its corresponding energy $<\frac1NS^{\frac{N}{2}}$ by using the Mountain Pass Theorem, where $S$ is the sharp constant of $D^{1,2}(\R^N)\hookrightarrow L^{2^*}(\R^N)$, i.e.
$$S=\inf\limits_{\substack{u\not\equiv0\\ u\in D^{1,2}(\R^N)}}\frac{\int_{\R^N}|\nabla u|^2}{(\int_{\R^N}|u|^{2^*})^{\frac{2}{2^*}}},$$
where $D^{1,2}(\R^N)\triangleq\{u\in L^{2^*}(\R^N)|~|\nabla u|\in L^2(\R^N)\}.$

In particular, in \eqref{1.5}, letting $f(u)=\la_1 |u|^{\al-2}u$ or $f(u)=\la_2|u|^{r-2}u$ with $\la_1,\la_2>0,2<\al,r<2^*$, then we conclude from \cite{dyb} that the following two equations
\begin{equation}\label{1.6}
 -\Delta u+u=|u|^{2^*-2}u+\la_1|u|^{\al-2}u,\,\,x\in \R^N
\end{equation}
and
\begin{equation}\label{1.7}
 -\Delta u+u=|u|^{2^*-2}v+\la_2|u|^{r-2}u,\,\,x\in \R^N
\end{equation}
respectively have at least one positive radial solution, denoted by $u_1,$ $v_1$. Moreover, their corresponding energy respectively satisfies that
\begin{equation}\label{1.8}
B_1\triangleq \ds (\frac12-\frac{1}{2^*})\int_{\R^N}|u_1|^{2^*}+(\frac12-\frac1\al)\la_1\int_{\R^N}|u_1|^\al<\frac1N S^{\frac N2}
\end{equation}
and
\begin{equation}\label{1.9}
B_2\triangleq \ds (\frac12-\frac{1}{2^*})\int_{\R^N}|v_1|^{2^*}+(\frac12-\frac1r)\la_2\int_{\R^N}|v_1|^r<\frac1N S^{\frac N2}.
\end{equation}

Based on the above papers, an interesting question is:  whether we can extend the existence results of \eqref{1.5} to system \eqref{1.1}.  In this paper, we will mainly discuss the existence of positive solutions to \eqref{1.1} in $\R^N$ with $N\geq5$, and obtain an affirmative answer.
As far as we know, there is no existence result for \eqref{1.1}.

Define $H\triangleq H^1_r(\R^N)\times H^1_r(\R^N)$ with  the norm
$$
\|(u\ ,v)\|=\Big[\int_{\R^N}(|\nabla u|^2+|u|^2)+\int_{\R^N}(|\nabla v|^2+|v|^2)\Big]^{\frac12},
$$
where $H_r^1(\R^N)\triangleq\{u\in H^1(\R^N):u(x)=u(|x|)\}$. It is well known that weak solutions of \eqref{1.1} correspond to critical points of the energy functional $I:H\rightarrow\R$ defined as follows
$$
I(u,v)=\frac{1}{2}\|(u,v)\|^2-\frac{1}{2^*}\ds\int_{\R^N}(|u|^{2^*}+|v|^{2^*}+2\beta|u|^{\frac{2^*}{2}}
|v|^{\frac{2^*}{2}})-\frac{\la_1}{\al}\int_{\R^N}|u|^\al-\frac{\la_2}{r}\int_{\R^N}|v|^r,
$$
for any $(u,v)\in H.$ We say $(u,v)\in H$ a positive solution of \eqref{1.1} if $(u,v)$ is a solution of \eqref{1.1} and $u>0,v>0.$

To state our main results, we set
$$
\begin{array}{ll}
M=\Big\{&\ds (u,v)\in H|~u\not\equiv0,~v\not\equiv0,~\int_{\R^N}(|\nabla u|^2+|u|^2)=\int_{\R^N}(|u|^{2^*}+\beta|u|^{\frac{2^*}{2}}
|v|^{\frac{2^*}{2}}+\la_1|u|^\al),\\
   &~~\ds \int_{\R^N}(|\nabla v|^2+|v|^2)=\int_{\R^N}(|v|^{2^*}+\be |u|^{\frac{2^*}{2}}
|v|^{\frac{2^*}{2}}+\la_2|v|^r)\Big\}.
 \end{array}
$$
Then $M\neq\emptyset$. In fact, take $\varphi,\psi\in C_0^{\infty}(\R^N)$ with $\varphi,\psi\not\equiv0$ and $supp(\varphi)\cap supp(\psi)=\emptyset$, then there exist $t_1,t_2>0$ such that $(t_1\varphi,t_2\psi)\in M$ since $\al,r>2$. So $M\neq\emptyset$.

Denote
$$
B\triangleq \ds \inf_{(u,v)\in M}I(u,v).
$$
It is easy to see that $B>0$ by the Sobolev embedding inequality.\\

Our main results are as follows:
\begin{theorem}\label{th1.1} If \,$N\geq5, \la_1,\la_2>0,2<\al,r<2^*$, then problem \eqref{1.1} has a positive solution $(u,v)\in H$ with $I(u,v)=B$ for any $\be>0$.
\end{theorem}

\begin{theorem}\label{th1.2} If \,$N\geq5, \la_1,\la_2>0,2<\al,r<2^*$, then problem \eqref{1.1} has a positive solution $(u,v)\in H$ with $I(u,v)=B$ for any $-\frac12\leq\be<0$.
\end{theorem}

\begin{remark}\label{re1.2} By the definition of $M$, we see that the solutions obtained in Theorems 1 and 2 are least energy solutions in the radially symmetric Sobolev subspace $H=H_r^1(\R^N)\times H_r^1(\R^N)$, however, we do not know that whether the solutions we obtained are the least energy solutions in the whole space $H^1(\R^N)\times H^1(\R^N).$
\end{remark}

We try to use the Mountain Pass Theorem and Ekeland's variational principle to prove Theorems 1.1 and 1.2. Our methods are inspired by the work of \cite{cz1,cz2}, which deals with problems in bounded domains. However, since we deal with problems in $\R^N$ and the appearance of two perturbation terms $\la_1 u^{\al-1}$ and $\la_2 v^{r-1}$ in \eqref{1.1}, we will encounter some new
difficulties. First, as we know, a substantial difference between a bounded domain and the whole space $\R^N$ is that the Sobolev embedding $H^1(\R^N)\hookrightarrow L^q(\R^N)$ with $2\leq q\leq 2^*$ is not compact. But Strauss' radial Lemma tells us that the embedding $H_r^1(\R^N)\hookrightarrow L^q(\R^N), 2<q<2^*$ is compact. Thus in this paper, we will discuss problem \eqref{1.1} in $H=H_r^1(\R^N)\times H_r^1(\R^N)$ to recover the lack of compactness. Secondly, since the nonlinearity and coupling terms are of critical growth, we still face the difficulty that $H_r^1(\R^N)\hookrightarrow L^{2^*}(\R^N)$ is not compact. For the case where $\be>0$, it easily see that the functional $I$ possesses a mountain pass structure around $0\in H$, then by the Mountain Pass Theorem, $I$ has a critical point, denoted by $(u,v)$, in $H$. But $(u,v)$ may be a trivial or semi-trivial critical point of $I$. In order to prove that $(u,v)$ is nontrivial, inspired by \cite{bn,cz1,cz2}, we try to pull the energy level down below some
critical level to recover certain compactness condition. However, problem \eqref{1.1} is much more complicated comparing with \eqref{1.4} due to the effect of the perturbation terms, so some new ideas are needed in the process of pulling down the energy level;
Meanwhile, In the case where $\beta<0$, we try to follow the method used in \cite{cz2} to prove the existence of positive solutions, i.e. we try to obtain a minimizing sequence of $B$ by Ekeland's variational principle and then to show that the minimizing sequence weakly converges to a critical point of $I$ and finally to prove that the critical point is nontrivial by pulling down the energy level. However, the method used in \cite{cz2} does not work here for all $\be<0$ because of the existence of the perturbation terms and need to be improved. We succeeded in doing so by proving the existence result for $\be<0$ restricted in some suitable interval and more careful analysis.

The paper is organized as follows. In Section 2,  we will prove Theorem 1.1; In Section 3, we will give the proof of Theorem 1.2. Throughout this paper, for simplicity, we denote various positive constants as $C$ and omit
$\,dx\,$ in integration. We use ``$\rightarrow,\ \rightharpoonup$" to denote the
strong and weak convergence in the related function space
respectively and  denote $B_r(x)\triangleq\{y\in\R^N|\,|x-y|<r\}$.

\section{ Proof of Theorem 1.1}
Define
$$
\ds \mathcal{B}= \inf_{h\in \Gamma}\max_{t>0}I(h(t)),
$$
where $\Gamma=\{h\in C([0,1],H)|~h(0)=0,I(h(t))<0\}.$ Since $2<\al,r<2^*$, it is easy to check that $\mathcal{B}$ is well defined and that
$$
\ds \mathcal{B}=\inf_{\substack{(u,v)\neq(0,0)\\ (u,v)\in H}}\max_{t>0}I(tu,tv)=\inf_{(u,v)\in \mathcal{M}}I(u,v),
$$
where $\mathcal{M}$ is the Nehari manifold associated to $I$, i.e.
$$
\begin{array}{ll}
\mathcal{M}=\Big\{&\ds (u,v)\in H\setminus \{(0,0)\}|~G(u,v)\triangleq\int_{\R^N}(|\nabla u|^2+|u|^2+|\nabla v|^2+|v|^2)-\int_{\R^N}(|u|^{2^*}\\[5mm]
&~~\ds +|v|^{2^*}+2\be |u|^{\frac{2^*}{2}}|v|^{\frac{2^*}{2}})
  -\int_{\R^N}(\la_1|u|^\al+\la_2|v|^r)=0\Big\}.
 \end{array}
 $$
Note that $M\subset \mathcal{M}$, one has that $\ds \mathcal{B}\leq B$ and $\mathcal{B}>0.$

Since the nonlinearity and the coupling terms are both of critical growth in \eqref{1.1}, the existence of nontrivial solutions to problem \eqref{1.1} depends heavily
on that to the following limiting problem
\begin{equation}\label{2.1}
 \left\{
 \begin{array}{ll}
 \ds -\Delta u=|u|^{2^*-2}u+\be |u|^{\frac{2^*}{2}-2}u|v|^{\frac{2^*}{2}-1},\, &x\in \R^N, \vspace{0.2cm}\\
 -\Delta v=|v|^{2^*-2}v+\be |u|^{\frac{2^*}{2}-1}|v|^{\frac{2^*}{2}-2}v,\, &x\in \R^N, \vspace{0.2cm}\\
 u,v\,\in D^{1,2}(\R^N).
 \end{array}
 \right.
\end{equation}
Define $D\triangleq D^{1,2}(\R^N)\times D^{1,2}(\R^N)$ and a $C^1$
functional $E:D\rightarrow \R$ given by
$$
E(u,v)=\ds \frac{1}{2}\int_{\R^N}(|\nabla u|^2+|\nabla v|^2)-\frac{1}{2^*}\ds\int_{\R^N}(|u|^{2^*}+|v|^{2^*}+2\be|u|^{\frac{2^*}{2}}
|v|^{\frac{2^*}{2}}).
$$
Set
$$
\begin{array}{ll}
\mathcal{N}=\Big\{&\ds (u,v)\in D|~u\not\equiv0,v\not\equiv0,\int_{\R^N}|\nabla u|^2=\int_{\R^N}(|u|^{2^*}+\be |u|^{\frac{2^*}{2}}|v|^{\frac{2^*}{2}}),\\[5mm]
   &~~\ds \int_{\R^N}|\nabla v|^2=\int_{\R^N}(|v|^{2^*}+\be |u|^{\frac{2^*}{2}}|v|^{\frac{2^*}{2}})\Big\}
 \end{array}
 $$
 and
 $$
\begin{array}{ll}
\mathcal{N'}=\Big\{&\ds (u,v)\in D\setminus \{(0,0)\}|~ \int_{\R^N}(|\nabla u|^2+|\nabla v|^2)=\ds \int_{\R^N}(|u|^{2^*}+|v|^{2^*}+2\be |u|^{\frac{2^*}{2}}|v|^{\frac{2^*}{2}})\Big\},
 \end{array}
 $$
i.e. $\mathcal{N'}$ is the Nehari manifold associated to $E$. Similarly to $M$, we have that $\mathcal{N}\neq\emptyset$ and $\mathcal{N}\subset \mathcal{N'}.$

Denote
$$
A\triangleq \ds \inf_{(u,v)\in \mathcal{N}}E(u,v),\,\,\,\,A'\triangleq \ds \inf_{(u,v)\in \mathcal{N'}}E(u,v).
$$
Then $0<A'\leq A.$

\begin{lemma}\label{le2.1}(\cite{cz2}, Theorem 1.6, Proposition 2.1)\\
$(i)$ If $\be<0,$ then $A$ is not attained and $A=\ds \frac 2NS^{\frac N2}$.\\
$(ii)$ If $\be>0$, then \eqref{2.1} has a positive least energy solution $(U,V)\in D$ with $E(U,V)=A$, which is radially symmetric decreasing. Moreover,
$$
U(x)+V(x)\leq C(1+|x|)^{2-N},\,\,|\nabla U|+|\nabla V|\leq C(1+|x|)^{1-N}
$$
\end{lemma}
and
$$A=A'.$$
\begin{lemma}\label{le2.2}(\cite{cz2}, Lemma 3.3)~~Let $u_n\rightharpoonup u, v_n\rightharpoonup v$ in $H^1(\R^N)$ as $n\rightarrow +\infty$, then passing to a subsequence, there holds
$$
\ds\lim_{n\rightarrow +\infty}\int_{\R^N}(|u_n|^{\frac{2^*}{2}}|v_n|^{\frac{2^*}{2}}-|u_n-u|^{\frac{2^*}{2}}|v_n-v|^{\frac{2^*}{2}}-|u|^{\frac{2^*}{2}}|v|^{\frac{2^*}{2}})=0.
$$
\end{lemma}

\begin{lemma}\label{le2.3}~~
Let $\be>0$, then
$$
\ds \mathcal{B}<\min\{B_1,B_2,A\},
$$
where $B_1,B_2$ are given in \eqref{1.8},\eqref{1.9}.
\end{lemma}
\begin{proof}~~The arguments follow from that of Lemma 3.4 in \cite{cz2}, however, due to the effect of the perturbation terms, it is more complicated and some new ideas are needed. The proof consists of two steps.

{\textbf Step 1:} we prove that $\mathcal{B}<A.$

For $\rho>0$, let $\psi\in C_0^{\infty}(B_{2\rho}(0))$ be a cut-off function with $0\leq\psi\leq1$ and $ \psi\equiv1 $ for $|x|\leq\rho.$ Recall that $(U,V)$ given in
Lemma~\ref{le2.1} (ii), for $\va>0$, we define
$$
(U_\va(x),V_\va(x))\triangleq \Big(\va^{-\frac{N-2}{2}}U(\frac{x}{\va}),\,\va^{-\frac{N-2}{2}}V(\frac{x}{\va})\Big).
$$
By direct calculation, then
$$
\ds \int_{\R^N}|\nabla U_\va|^2=\int_{\R^N}|\nabla U|^2,\,\,\int_{\R^N}|U_\va|^{2^*}=\int_{\R^N}|U|^{2^*},
$$
$$
\ds \int_{\R^N}|\nabla V_\va|^2=\int_{\R^N}|\nabla V|^2,\,\,\int_{\R^N}|V_\va|^{2^*}=\int_{\R^N}|V|^{2^*}.
$$
Denote
$$
(u_\va,v_\va)\triangleq (\psi U_\va, \psi V_\va).
$$
By the estimates given in Lemma 3.4 of \cite{cz2}, then we see that
\begin{equation}\label{2.2}
 \ds \int_{\R^N}|\nabla u_\va|^2\leq\int_{\R^N}|\nabla U|^2+O(\va^{N-2}),
\end{equation}
\begin{equation}\label{2.3}
 \ds \int_{\R^N}|u_\va|^{2^*}\geq\int_{\R^N}|U|^{2^*}+O(\va^N),
\end{equation}
\begin{equation}\label{2.4}
 \ds \int_{\R^N}|u_\va|^{\frac{2^*}{2}}|v_\va|^{\frac{2^*}{2}}\geq\int_{\R^N}|U|^{\frac{2^*}{2}}|V|^{\frac{2^*}{2}}+O(\va^N).
\end{equation}
Moreover, we still have the following ineuqality:
\begin{equation}\label{2.5}
 \ds \int_{\R^N}|u_\va|^2\leq C\va^2,
\end{equation}
\begin{equation}\label{2.6}
 \ds \int_{\R^N}|u_\va|^\al\geq C\va^{N-\frac{N-2}{2}\al}+O(\va^{\frac{N-2}{2}\al}),
\end{equation}
where $C$ denotes a positive constant. In fact, let $0<\va\ll \rho,$
\begin{eqnarray*}
 \ds\int_{\R^N}|u_\va|^2
&\leq &\ds \int_{|x|\leq 2\rho}\va^{2-N}U^2(\frac{x}{\va})
= \ds \va^2 \int_{|x|\leq \frac{2\rho}{\va}}U^2(x)\\
&\leq& \va^2\int_{\R^N}U^2(x)\leq C\va^2.
\end{eqnarray*}
By Lemma~\ref{le2.1}(ii), we have that
\begin{eqnarray*}
 \ds\int_{\R^N}|u_\va|^\al
&\geq &\ds \int_{|x|\leq \rho}\va^{-\frac{N-2}{2}\al}|U(\frac{x}{\va})|^\al
= \ds \int_{|x|\leq \frac{\rho}{\va}}\va^{N-\frac{N-2}{2}\al}|U(x)|^\al\\
&=&\va^{N-\frac{N-2}{2}\al}(\int_{\R^N}|U(x)|^\al-\int_{|x|\geq\frac{\rho}{\va}}|U(x)|^\al)\\
&\geq&C\va^{N-\frac{N-2}{2}\al}-C\va^{N-\frac{N-2}{2}\al}\int_{\frac{\rho}{\va}}^{\infty}r^{N-1-(N-2)\al}dr\\
&=&C\va^{N-\frac{N-2}{2}\al}+O(\va^{\frac{N-2}{2}\al}).
\end{eqnarray*}
Similarly,
\begin{equation}\label{2.7}
 \ds \int_{\R^N}|\nabla v_\va|^2\leq\int_{\R^N}|\nabla V|^2+O(\va^{N-2}),
\end{equation}
\begin{equation}\label{2.8}
 \ds \int_{\R^N}|v_\va|^{2^*}\geq\int_{\R^N}|V|^{2^*}+O(\va^N),
\end{equation}
\begin{equation}\label{2.9}
 \ds \int_{\R^N}|v_\va|^2\leq C\va^2,
\end{equation}
\begin{equation}\label{2.10}
 \ds \int_{\R^N}|v_\va|^r\geq C\va^{N-\frac{N-2}{2}r}+O(\va^{\frac{N-2}{2}r}),
\end{equation}
where $C$ denotes a positive constant. Recall that $E(U,V)=A$ and $(U,V)\in \mathcal{N}$, then
\begin{equation}\label{2.11}
\ds NA=\int_{\R^N}|\nabla U|^2+|\nabla V|^2=\int_{\R^N}(|U|^{2^*}+|V|^{2^*}+2\be |U|^{\frac{2^*}{2}}|V|^{\frac{2^*}{2}}).
\end{equation}
For any $t>0$, set
$$
\begin{array}{ll}
g_\va(t)\triangleq I(tu_\va,tv_\va)=&\ds \frac{t^2}{2} \|(u_\va,v_\va)\|^2-\frac{t^{2^*}}{2^*}\int_{\R^N}(|u_\va|^{2^*}+|v_\va|^{2^*}+2\be |u_\va|^{\frac{2^*}{2}}|v_\va|^{\frac{2^*}{2}})\\
&~~~~~~~~~~~~~~~~~~~~~~~~~~~~~~~~~~~~~~~~~~~~~~~~~~~~~~~~~~~~~~\ds -\frac{\la_1 t^\al}{\al}\int_{\R^N}|u_\va|^\al-\frac{\la_2 t^r}{r}\int_{\R^N}|v_\va|^r.
\end{array}$$
Since $\al,r>2$, we easily see that $g_\va(t)$ has a unique critical point $t_{\va}\triangleq t_{u_\va,v_\va}>0$, which corresponds to its maximum, i.e.
$$
I(t_{\va}u_\va,t_{\va}v_\va)=\max\limits_{t>0}I(tu_\va,tv_\va).
$$
By $g_\va'(t_{\va})=0$, then
\begin{equation}\label{2.12}
(t_{\va}u_\va,t_{\va}v_\va)\in \mathcal{M}.
\end{equation}
 We claim that $\{t_\va\}_{\va>0}$ is bounded from below by a positive constant. Otherwise, there exists a sequence
$\{\va_n\}\subset\R_+$ satisfying
$\lim\limits_{n\rightarrow\infty}t_{\va_n}=0$ and
$I(t_{\va_n}u_{\va_n},t_{\va_n}v_{\va_n})=\max\limits_{t>0}
I(tu_{\va_n},tv_{\va_n})$, then $0<
\mathcal{B}\leq\lim\limits_{n\rightarrow\infty}I(t_{\va_n}u_{\va_n},t_{\va_n}v_{\va_n})=0$,
which is impossible. So there exists $C>0$ independent of
$\varepsilon$ satisfying
\begin{equation}\label{2.13}
t_\varepsilon>C>0~~\hbox{for~all}~\varepsilon>0.
\end{equation}
Since $N\geq 5$ and $2<\al,r<2^*$,
\begin{equation}\label{2.23}
0<N-\frac{N-2}{2}\al<2<N-2<\frac{N-2}2\al<N
\end{equation}
and
\begin{equation}\label{2.24}
0<N-\frac{N-2}{2}r<2<N-2<\frac{N-2}{2}r<N.
\end{equation}
Then by \eqref{2.2}-\eqref{2.24}, we see that
\begin{eqnarray*}
 \ds I(t_\va u_\va,t_\va v_\va)
&=&\ds \frac{t_\va^2}{2}\|(u_\va,v_\va)\|^2-\frac{t_\va^{2^*}}{2^*}\ds\int_{\R^N}(|u_\va|^{2^*}+|v_\va|^{2^*}+2\be |u_\va|^{\frac{2^*}{2}}
|v_\va|^{\frac{2^*}{2}})\\
&&-\frac{\la_1t_\va^\al}{\al}\int_{\R^N}|u_\va|^\al-\frac{\la_2t_\va^r}{r}\int_{\R^N}|v_\va|^r\\
&\leq&\frac{t_\va^2}{2}[\int_{\R^N}(|\nabla U|^2+|\nabla V|^2)+C\va^2+O(\va^{N-2})]-\frac{t_\va^{2^*}}{2^*}\ds[\int_{\R^N}(|U|^{2^*}|+|V|^{2^*}\\
&&+2\beta|U|^{\frac{2^*}{2}}
|V|^{\frac{2^*}{2}})+O(\va^N)]-\frac{\la_1t_\va^\al}{\al}[C\va^{N-\frac{N-2}{2}\al}+O(\va^{\frac{N-2}{2}\al})]\\
&&-\frac{\la_2t_\va^r}{r}[C\va^{N-\frac{N-2}{2}r}+O(\va^{\frac{N-2}{2}r})]\\
&\leq&\frac{t_\va^2}{2}[NA+C\va^2+O(\va^{N-2})]-\frac{t_\va^{2^*}}{2^*}[NA+O(\va^N)]\\
&&-\frac{\la_1C}{\al}[C\va^{N-\frac{N-2}{2}\al}+O(\va^{\frac{N-2}{2}\al})]-\frac{\la_2C}{r}[C\va^{N-\frac{N-2}{2}r}+O(\va^{\frac{N-2}{2}r})]\\
&\leq&\frac{1}{N}[NA+C\va^2+O(\va^{N-2})]\left(\frac{NA+C\va^2+O(\va^{N-2})}{NA+O(\va^N)}\right)^{\frac{N-2}{2}}\\
&&-[C\va^{N-\frac{N-2}{2}\al}+O(\va^{\frac{N-2}{2}\al})]-[C\va^{N-\frac{N-2}{2}r}+O(\va^{\frac{N-2}{2}r})]\\
&\leq&A-C\va^{N-\frac{N-2}{2}\al}-C\va^{N-\frac{N-2}{2}r}+O(\va^{\frac{N-2}{2}\al})+O(\va^{\frac{N-2}{2}r})+C\va^2+O(\va^{N-2})\\
&<& A~~~~~~~\ \ \ \ \ ~~~\text{for}~\va>0~\text{small enough}.
\end{eqnarray*}
Hence, by \eqref{2.12}, there holds $\ds \mathcal{B}<A$ for $\va>0$ small enough.

{\textbf Step 2.} we prove that $\ds \mathcal{B}<B_1.$

For $u_1,v_1$ given in \eqref{1.6}\eqref{1.7}, we define a function $F:\R^2\rightarrow\R^1$ by
$$F(t,s)=\langle I^\prime(tu_1,tsv_1),(tu_1,tsv_1)\rangle.$$
By the definition of $u_1$, it is easy to check that $F(1,0)=0$ and $\frac{\partial F(t,s)}{\partial t}(1,0)\neq0$. Then by the Implicit Function Theorem, there exists a $\delta>0$ and a function $t(s)$ such that
$$t(s)\in C^1(-\delta,\delta),\ \ t(0)=1,\ \ t^\prime(s)=\ds -\frac{F_s(t,s)}{F_t(t,s)}$$
and
$$F(t(s),s)=0,~~\forall~s\in(-\delta,\delta),$$
which implies that
\begin{equation}\label{2.14}
\ds (t(s)u_1,t(s)sv_1)\in \mathcal{M},~~\forall~s\in(-\delta,\delta).
\end{equation}

By direct calculation, we obtain that
\begin{eqnarray*}
 \ds F_s(t,s)&=&2st^2\int_{\R^N}(|\nabla v_1|^2+|v_1|^2)-2^*t^{2^*}s^{2^*-1}\int_{\R^N}|v_1|^{2^*}-
2^*\be t^{2^*}s^{\frac{2^*}{2}-1}\int_{\R^N}|u_1|^{\frac{2^*}{2}}|v_1|^{\frac{2^*}{2}}\\
&&-\la_2rs^{r-1}t^r\int_{\R^N}|v_1|^r
\end{eqnarray*}
and
\begin{eqnarray*}
 \ds F_t(t,s)&=&2t\|(u_1,s v_1)\|^2-2^*t^{2^*-1}[\int_{\R^N}|u_1|^{2^*}+s^{2^*}\int_{\R^N}|v_1|^{2^*}+
2\be s^{\frac{2^*}{2}}\int_{\R^N}|u_1|^{\frac{2^*}{2}}|v_1|^{\frac{2^*}{2}}]\\
&&-\al\la_1 t^{\al-1}\int_{\R^N}|u_1|^\al-r\la_2t^{r-1}s^r\int_{\R^N}|v_1|^r.
\end{eqnarray*}
Since $N\geq 5$ and $r\in (2,2^*)$, $\frac{2^*}{2}<2<r<2^*$, then we have that
$$
\ds\lim_{s\rightarrow 0}\frac{t'(s)}{|s|^{\frac{2^*}{2}-2}s}=\lim_{s\rightarrow 0}\frac{-F_s/F_t}{|s|^{\frac{2^*}{2}-2}s}=\ds\frac{-2^*\be \int_{\R^N}|u_1|^{\frac{2^*}{2}}|v_1|^{\frac{2^*}{2}}}{(2^*-2)\int_{\R^N}|u_1|^{2^*}
+\la_1(\al-2) \int_{\R^N}|u_1|^\al},
$$
i.e.
$$
\ds t'(s)=\ds\frac{-2^*\be \int_{\R^N}|u_1|^{\frac{2^*}{2}}|v_1|^{\frac{2^*}{2}}}{(2^*-2) \int_{\R^N}|u_1|^{2^*}
+\la_1(\al-2)\int_{\R^N}|u_1|^\al}|s|^{\frac{2^*}{2}-2}s(1+o(1))~~\text{as}~s\rightarrow 0.
$$
So
$$
\ds t(s)=\ds 1-\frac{2\be \int_{\R^N}|u_1|^{\frac{2^*}{2}}|v_1|^{\frac{2^*}{2}}}{(2^*-2) \int_{\R^N}|u_1|^{2^*}
+\la_1(\al-2)\ds \int_{\R^N}|u_1|^\al}|s|^{\frac{2^*}{2}}(1+o(1))~~\text{as}~s\rightarrow 0,
$$
which implies that
\begin{equation}\label{2.15}
\ds t^{2^*}(s)=\ds 1-\frac{22^*\be \int_{\R^N}|u_1|^{\frac{2^*}{2}}|v_1|^{\frac{2^*}{2}}}{(2^*-2) \int_{\R^N}|u_1|^{2^*}
+\la_1(\al-2) \int_{\R^N}|u_1|^\al}|s|^{\frac{2^*}{2}}(1+o(1))~~\text{as}~s\rightarrow 0,
\end{equation}
\begin{equation}\label{2.16}
\ds t^{\al}(s)=\ds 1-\frac{2\al\be \int_{\R^N}|u_1|^{\frac{2^*}{2}}|v_1|^{\frac{2^*}{2}}}{(2^*-2) \int_{\R^N}|u_1|^{2^*}
+\la_1(\al-2) \int_{\R^N}|u_1|^\al}|s|^{\frac{2^*}{2}}(1+o(1))~~\text{as}~s\rightarrow 0
\end{equation}
and
\begin{equation}\label{2.17}
\ds t^{r}(s)=\ds 1-\frac{2r\be \int_{\R^N}|u_1|^{\frac{2^*}{2}}|v_1|^{\frac{2^*}{2}}}{(2^*-2) \int_{\R^N}|u_1|^{2^*}
+\la_1(\al-2) \int_{\R^N}|u_1|^\al}|s|^{\frac{2^*}{2}}(1+o(1))~~\text{as}~s\rightarrow 0.
\end{equation}
Thus by \eqref{1.8},\eqref{2.14}-\eqref{2.17} and $\frac{2^*}{2}<2<r$, we see that for $\forall~s\in(-\delta,\delta)$,
\begin{eqnarray*}
 \ds \mathcal{B}
&\leq& I(t(s)u_1,t(s)sv_1)\\
&=&\ds \frac{t(s)^2}{2}\int_{\R^N}[(|\nabla u_1|^2+|u_1|^2)+s^2(|\nabla v_1|^2+|v_1|^2)]-\frac{t(s)^{2^*}}{2^*}\ds\int_{\R^N}(|u_1|^{2^*}+|s|^{2^*}|v_1|^{2^*}\\
&&+2\be |s|^{\frac{2^*}{2}}|u_1|^{\frac{2^*}{2}}|v_1|^{\frac{2^*}{2}})-
\frac{\la_1t(s)^\al}{\al}\int_{\R^N}|u_1|^\al-\frac{\la_2t(s)^r}{r}\int_{\R^N}|s|^r|v_1|^r\\
&=&\ds
(\frac12-\frac{1}{2^*})t(s)^{2^*}\int_{\R^N}(|u_1|^{2^*}+|s|^{2^*}|v_1|^{2^*}+2\be |s|^{\frac{2^*}{2}}|u_1|^{\frac{2^*}{2}}|v_1|^{\frac{2^*}{2}})\\
&&+(\frac12-\frac1\al)\la_1t(s)^\al\int_{\R^N}|u_1|^\al+(\frac12-\frac1r)\la_2t(s)^r\int_{\R^N}|s|^r|v_1|^r\\
&=&(\frac12-\frac{1}{2^*})\left[1-\frac{22^*\be \int_{\R^N}|u_1|^{\frac{2^*}{2}}|v_1|^{\frac{2^*}{2}}}{(2^*-2)\int_{\R^N}|u_1|^{2^*}
+\la_1(\al-2) \int_{\R^N}|u_1|^\al}|s|^{\frac{2^*}{2}}(1+o(1))\right]\\
&&\times\ds \int_{\R^N}(|u_1|^{2^*}+|s|^{2^*}|v_1|^{2^*}+2\be |s|^{\frac{2^*}{2}}|u_1|^{\frac{2^*}{2}}|v_1|^{\frac{2^*}{2}})\\
&&+\left[1-\frac{2\al\be \int_{\R^N}|u_1|^{\frac{2^*}{2}}|v_1|^{\frac{2^*}{2}}}{(2^*-2) \int_{\R^N}|u_1|^{2^*}
+\la_1(\al-2) \int_{\R^N}|u_1|^\al}|s|^{\frac{2^*}{2}}(1+o(1))\right]\\
&&~~\times(\frac12-\frac1\al)\la_1\int_{\R^N}|u_1|^\al\\
&&+\left[1-\frac{2r\be \int_{\R^N}|u_1|^{\frac{2^*}{2}}|v_1|^{\frac{2^*}{2}}}{(2^*-2) \int_{\R^N}|u_1|^{2^*}
+\la_1(\al-2) \int_{\R^N}|u_1|^\al}|s|^{\frac{2^*}{2}}(1+o(1))\right]\\
&&~~\times(\frac12-\frac1r)\la_2\int_{\R^N}|s|^r|v_1|^r\\
&\leq&\ds (\frac{1}{2}-\frac{1}{2^*})\int_{\R^N}|u_1|^{2^*}+(\frac12-\frac1\al)\la_1\int_{\R^N}|u_1|^\al
-C|s|^{\frac{2^*}{2}}+o(|s|^{\frac{2^*}{2}})\ \ (\exists~C>0)\\
&<&\ds (\frac12-\frac{1}{2^*})\int_{\R^N}|u_1|^{2^*}+(\frac12-\frac1\al)\la_1\int_{\R^N}|u_1|^\al=B_1,\ \ \ \ \text{as}~~|s|>0~\text{small~enough},
\end{eqnarray*}
where $C>0$ is a constant independent of $s$. Hence $\ds \mathcal{B}<B_1$.

Similarly, we have that $\ds \mathcal{B}<B_2.$ Hence the proof of the Lemma is completed.
\end{proof}

\noindent $\textbf{Proof of Theorem 1.1}$

Since $\be>0$, it is easy to check that $I(u,v)$ possesses a mountain pass structure, then by the Mountain Pass Theorem in \cite{ar,w1},
there exists a sequence $\{(u_n,v_n)\}\subset H$ such that
$$
\ds\lim_{n\rightarrow +\infty}I(u_n,v_n)=\mathcal{B},\,\,\lim_{n\rightarrow +\infty}I'(u_n,v_n)=0.
$$
It is standard to see that $\{(u_n,v_n)\}$ is bounded in $H$, so we may assume that $(u_n,v_n)\rightharpoonup (u,v)$ in $H$ for some $(u,v)\in H$. Set $\om_n=u_n-u,\sigma_n=v_n-v$. Then
$$\om_n\rightharpoonup 0, \ \ \ \sigma_n\rightharpoonup0\ \ \ \text{in}~H^1_r(\R^N),$$
$$\om_n\rightharpoonup 0, \ \ \ \sigma_n\rightharpoonup0\ \ \ \text{in}~L^{2^*}(\R^N)$$
and
$$\om_n\rightarrow 0\ \ \ \text{in}~L^{\al}(\R^N),\ \ \ \sigma_n\rightarrow0\ \ \ \text{in}~L^{r}(\R^N).$$
Hence by Brezis-Lieb lemma and Lemma~\ref{le2.2}, we have that
$I'(u,v)=0$ and
\begin{equation}\label{2.18}
\ds\int_{\R^N}|\nabla \omega_n|^2-\int_{\R^N}(|\om_n|^{2^*}+\be|\om_n|^{\frac{2^*}{2}}|\sigma_n|^{\frac{2^*}{2}})=o(1),
\end{equation}
\begin{equation}\label{2.19}
\ds\int_{\R^N}|\nabla \sigma_n|^2-\int_{\R^N}(|\sigma_n|^{2^*}+\be|\om_n|^{\frac{2^*}{2}}|\sigma_n|^{\frac{2^*}{2}})=o(1)
\end{equation}
and
\begin{equation}\label{2.20}
I(u_n,v_n)=I(u,v)+E(\om_n,\sigma_n)+o(1).
\end{equation}
Passing to a subsequence, we may assume that
$$
\ds\lim_{n\rightarrow +\infty}\int_{\R^N}|\nabla \om_n|^2=b_1,\,\,
\lim_{n\rightarrow +\infty}\int_{\R^N}|\nabla \sigma_n|^2=b_2.
$$
Then by \eqref{2.18},\eqref{2.19}, we have that
\begin{equation}\label{2.21}
E(\om_n,\sigma_n)=\frac1N(b_1+b_2)+o(1).
\end{equation}
Letting $n\rightarrow +\infty$ in \eqref{2.20}, we have that
\begin{equation}\label{2.22}
0\leq I(u,v)\leq I(u,v)+\frac1N(b_1+b_2)=\ds\lim_{n\rightarrow +\infty}I(u_n,v_n)=\mathcal{B}.
\end{equation}

\textbf{Case 1:} $u\equiv0,v\equiv0$.

By \eqref{2.22}, we have that $0<N\mathcal{B}=b_1+b_2<+\infty$, then we may assume that $(\om_n,\sigma_n)\neq(0,0)$ for $n$ large. By the definition of $\mathcal{N'}$ and \eqref{2.18},\eqref{2.19}, similar to the proof of \eqref{2.12}, we see that there exists a sequence $\{t_n\}\subset\R_+$ such that $(t_n\om_n,t_n\sigma_n)\in \mathcal{N'}$ and $t_n\rightarrow 1$ as $n\rightarrow +\infty$. Then by \eqref{2.18},\eqref{2.21} and Lemma 2.1 (ii), we have that
$$
\ds \mathcal{B}=\frac1N(b_1+b_2)=\lim_{n\rightarrow +\infty}E(\om_n,\sigma_n)=\lim_{n\rightarrow +\infty}E(t_n\om_n,t_n\sigma_n)\geq A'= A,
$$
which contradicts to Lemma~\ref{le2.3}. So Case 1 is impossible.

\textbf{Case 2:} $u\not\equiv 0,v\equiv0$ or $u\equiv0,v\not\equiv 0.$

Without loss of generality, we may assume that $u\not\equiv 0,v\equiv0$. Them $u\in H^1_r(\R^N)$ is a nontrivial solution of
$-\Delta u+u=|u|^{2^*-2}u+\la_1|u|^{\al-2}u$, and so $\ds \mathcal{B}\geq I(u,0)\geq B_1$, which contradicts to Lemma~\ref{le2.3}. so Case 2 is impossible.

Since Case 1 and Case 2 are both impossible, we see that $u\not\equiv 0,v\not\equiv 0$. By $I'(u,v)=0$, then we have that $(u,v)\in M$. By $\ds \mathcal{B}\leq B$ and \eqref{2.22}, we have that
$$\ds I(u,v)=\mathcal{B}=B.$$
Moreover, it is easy to see that $(|u|,|v|)\in M\subset \mathcal{M}$ and $I(|u|,|v|)=\mathcal{B}=B$ since the functional $I$ and the manifolds $M$ and $\mathcal{M}$ are symmetric, hence we may assume that such a minimizer of $B$ does not change sign, i.e. $u\geq0,v\geq0$. This means that $(u,v)\in H$ is a nontrivial nonnegative solution of \eqref{1.1} with $I(u,v)=B$. By the maximum principle, we see that $u(x),v(x)>0$ for all $x\in\R^N$. Thus, $(u,v)$ is a positive solution of \eqref{1.1} in $H$ with $I(u,v)=B$. This completes the proof of Theorem~\ref{th1.1}.

\section{ Proof of Theorem 1.2}
For $\va>0$ and $y\in \R^N$, the following Aubin-Talenti instanton $U_{\va,y}\in D^{1,2}(\R^N)$ (see \cite{at,tg})
$$
U_{\va,y}(x)=\ds [N(N-2)]^{\frac{N-2}{4}}\Big(\frac{\va}{\va^2+|x-y|^2}\Big)^{\frac{N-2}{2}}.
$$
Then $U_{\va,y}$ solves $-\Delta u=|u|^{2^*-2}u$ in $\R^N$ and
\begin{equation}\label{3.1}
\ds\int_{\R^N}|\nabla U_{\va,y}|^2=\int_{\R^N}|U_{\va,y}|^{2^*}=S^{\frac N2}.
\end{equation}
Furthermore, $\{U_{\va,y}:\va>0,y\in \R^N\}$ contains all positive solutions of the equation  $-\Delta u=|u|^{2^*-2}u$ in $\R^N$.

As has been mentioned in Section 1, $u_1$, $v_1\in H^1_r(\R^N)$ are positive least energy radial solutions of \eqref{1.6} and \eqref{1.7} respectively, moreover, by the standard regularity arguments, $u_1,v_1\in C(\R^N).$
\begin{lemma}\label{le3.1}~~
Suppose that  $-\frac12\leq\be<0$, then we have that
$$
\ds B<\min\{B_1+\frac1N S^{\frac N2},B_2+\frac1N S^{\frac N2},A\}.
$$
\end{lemma}
\begin{proof} The proof is similar to that of Lemma 3.1 in \cite{cz2}, however, due to the effect of the perturbation terms, the arguments need to be slightly improved. We give its detailed proof.

Let $-\frac{1}{2}\leq\be<0$, consider the following function
$$F(t)=\frac{t^2}{2}\int_{\R^N}(|\nabla u_1|^2+|u_1|^2)-\frac{t^{2^*}}{22^*}\int_{\R^N}|u_1|^{2^*}-
\frac{\la_1t^\al}{\al}\int_{\R^N}|u_1|^\al+\frac{2^{\frac{N-2}{2}}}{N}S^{\frac N2},\ \ t>0.$$
Then by $\al>2$, there exists a $t_0>0$ such that
\begin{equation}\label{3.2}
F(t)<0,\ \ \ \forall~t>t_0.
\end{equation}

For $y_0\in\R^N,R>0$, let
$\psi\in C_0^\infty(B_{2R}(y_0))$ be a cut-off function with $0\leq\psi\leq 1$ and $\psi\equiv 1$ for $|x-y_0|\leq R$. Define
$v_\va=\psi U_{\va,y_0}$, where $U_{\va,y_0}$ is defined in \eqref{3.1}. Then by \cite{bn}, we have the
following estimates:
\begin{equation}\label{3.3}
\ds\int_{\R^N}|\nabla v_\va|^2=S^{\frac N2}+O(\va^{N-2}),\ \ \ \int_{\R^N}|v_\va|^{2^*}=S^{\frac N2}+O(\va^N),\,\,
\end{equation}
\begin{equation}\label{3.5}
\ds\int_{\R^N}| v_\va|^2\leq C\va^2,\ \ \
\int_{\R^N}| v_\va|^r\geq C\va^{N-\frac{N-2}{2}r}+O(\va^{\frac{N-2}{2}r})
\end{equation}
and
\begin{equation}\label{3.6}
\begin{array}{ll}
\ds \int_{\R^N}|v_\va|^{\frac{2^*}{2}}&\leq
\ds\int_{B(y_0,2R)}|U_{\va, y_0}|^{\frac{2^*}{2}}\leq  \ds C\int_{B(0,2R)}\Big(\frac{\va}{\va^2+|x|^2}\Big)^{\frac N2}\\
&\leq \ds C\va^{\frac N2}(\ln{\frac{2R}{\va}}+1)=o(\va^2),
\end{array}
\end{equation}
where $C>0$ is a constant.

Since $|\be|\leq\frac12$, we have for any $t,s>0$ that
$$
\ds 2|\be|t^{\frac{2^*}{2}}s^{\frac{2^*}{2}}\int_{\R^N}|u_1|^{\frac{2^*}{2}}|v_\va|^{\frac{2^*}{2}}
\leq \frac{t^{2^*}}{2}\int_{\R^N}|u_1|^{2^*}+\frac{s^{2^*}}{2}\int_{\R^N}|v_\va|^{2^*}
$$
and then
\begin{equation}\label{3.8}
\begin{array}{ll}
 \ds &I(tu_1,sv_\va)\\
=&\ds \frac{t^2}{2}\int_{\R^N}(|\nabla u_1|^2+|u_1|^2)+\frac{s^2}{2}\int_{\R^N}(|\nabla v_\va|^2+|v_\va|^2)\\
&-\ds \frac{1}{2^*}\int_{\R^N}(t^{2^*}|u_1|^{2^*}+s^{2^*}|v_1|^{2^*}+2\be t^{\frac{2^*}{2}}s^{\frac{2^*}{2}}|u_1|^{\frac{2^*}{2}}|v_\va|^{\frac{2^*}{2}})\\
&-\ds\frac{\la_1t^\al}{\al}\int_{\R^N}|u_1|^\al-\frac{\la_2s^r}{r}\int_{\R^N}|v_\va|^r\\
\leq&\ds\Big[\frac{t^2}{2}\int_{\R^N}(|\nabla u_1|^2+|u_1|^2)-\frac{t^{2^*}}{22^*}\int_{\R^N}|u_1|^{2^*}-\frac{\la_1t^\al}{\al}\int_{\R^N}|u_1|^\al\Big]
\\
&+\ds \Big[\frac{s^2}{2}\int_{\R^N}(|\nabla v_\va|^2+|v_\va|^2)
-\frac{s^{2^*}}{22^*}\int_{\R^N}|v_\va|^{2^*}-\frac{\la_2s^r}{r}\int_{\R^N}|v_\va|^r\Big]\\
\triangleq &f(t)+g(s).
\end{array}
\end{equation}
Similar to the proof of \eqref{2.13}, since $r>2$,
$$
\ds g(s)=\frac{s^2}{2}\int_{\R^N}(|\nabla v_\va|^2+|v_\va|^2)
-\frac{s^{2^*}}{22^*}\int_{\R^N}v_\va^{2^*}-\frac{\la_2s^r}{r}\int_{\R^N}|v_\va|^r,~~~s>0,
$$
has a unique critical point $s_\va>0$ corresponding to its maximum and there exists $C>0$ independent of $\va$ such that
\begin{equation}\label{3.7}
s_\va>C>0\ \ \text{for~all}~\va>0.
\end{equation}
Then by \eqref{3.3},\eqref{3.5},\eqref{3.7} and $r>2$, by direct calculation, it is easy to check that
\begin{equation}\label{3.26}
\ds \max_{s>0}g(s)=g(s_\va)<\frac1N 2^{\frac{N-2}{2}}S^{\frac N2}\ \ \ \hbox{for}~\va~\hbox{small~enough}.
\end{equation}
Then by \eqref{3.2},
$$f(t)+g(s)<0,\ \ \ \forall~t>t_0,~s>0.$$
Thus by \eqref{3.8},
$$
\ds\max_{t,s>0}I(tu_1,sv_\va)=\max_{0<t\leq t_0,s>0}I(tu_1,sv_\va).
$$
Set
$$
\ds g_\va(s)=\frac{s^2}{2}\int_{\R^N}(|\nabla v_\va|^2+|v_\va|^2)
-\frac{s^{2^*}}{2^*}\int_{\R^N}v_\va^{2^*}-\frac{\la_2s^r}{r}\int_{\R^N}|v_\va|^r,~~~s>0.
$$
Similar to the proof of \eqref{3.7}, $g_\va(s)$ has a unique critical point $s(\va)>0$ such that $$g_\va(s(\va))=\max\limits_{s>0}g_\va(s)\ \ \ \text{and}\ \ \  s(\va)>C_0>0,$$
where $C_0>0$ is a constant independent of $\va$. Since $g_\va(s)$ is strictly increasing for $0<s\leq s(\va)$, for any $0<s<C_0$, we have that $g_\va(s)<g_\va(C_0)$ and then
$$I(tu_1,sv_\va)<I(tu_1,C_0v_\va),\ \ \forall~t>0,~0<s<C_0.$$
Hence
\begin{equation}\label{3.9}
\ds \max_{t,s>0}I(tu_1,sv_\va)=\max_{0<t<t_0,s\geq C_0}I(tu_1,sv_\va).
\end{equation}
For $0<t<t_0,s\geq C_0$, we see from \eqref{3.6}, $-\frac{1}{2}\leq\be<0$ and $u_1\in C(\R^N)$ that there exists a $C\triangleq\max\limits_{B_{y_0}(R_0)}u_1>0$ such that
$$
\ds |\be|t^{\frac{2^*}{2}}s^{\frac{2^*}{2}}\int_{\R^N}|u_1|^{\frac{2^*}{2}}|v_\va|^{\frac{2^*}{2}}
\leq CC_0^{\frac{2^*}{2}-2}\frac{t_0^{\frac{2^*}{2}}}{2}s^2\int_{\R^N}|v_\va|^{\frac{2^*}{2}}\leq s^2o(\va^2),
$$
so
\begin{equation}\label{3.10}
\begin{array}{ll}
 \ds &I(tu_1,sv_\va)\\
=&\ds \frac{t^2}{2}\int_{\R^N}(|\nabla u_1|^2+|u_1|^2)+\frac{s^2}{2}\int_{\R^N}(|\nabla v_\va|^2+|v_\va|^2)\\
&-\ds \frac{1}{2^*}\int_{\R^N}(t^{2^*}|u_1|^{2^*}+s^{2^*}|v_\va|^{2^*}+2\be t^{\frac{2^*}{2}}s^{\frac{2^*}{2}}|u_1|^{\frac{2^*}{2}}|v_\va|^{\frac{2^*}{2}})\\
&-\ds\frac{\la_1t^\al}{\al}\int_{\R^N}|u_1|^\al-\frac{\la_2s^r}{r}\int_{\R^N}|v_\va|^r\\
\leq&\ds \Big[\frac{t^2}{2}\int_{\R^N}(|\nabla u_1|^2+|u_1|^2)-\frac{t^{2^*}}{2^*}\int_{\R^N}|u_1|^{2^*}-\frac{\la_1t^\al}{\al}\int_{\R^N}|u_1|^\al\Big]\\
&+\ds \Big[\frac{s^2}{2}\Big(\int_{\R^N}(|\nabla v_\va|^2+|v_\va|^2)+o(\va^2)\Big)-\frac{s^{2^*}}{2^*}\int_{\R^N}|v_\va|^{2^*}-\frac{\la_2s^r}{r}\int_{\R^N}|v_\va|^r\Big]\\
\triangleq&f_1(t)+g_1(s).
\end{array}
\end{equation}
Note that $\ds \max_{t>0}f_1(t)=f_1(1)=B_1$. Similar to the proof of \eqref{3.26}, we have that
$$
\ds \max_{s>0}g_1(s)<\frac 1NS^{\frac N2}\ \ \ \hbox{for}~\va~\hbox{small~enough}.
$$
By \eqref{3.9}\eqref{3.10}, we see that
\begin{equation}\label{3.11}
\begin{array}{ll}
\ds \max_{t,s>0}I(tu_1,sv_\va)&=\ds\max_{0<t\leq t_0,s\geq 1}I(tu_1,sv_\va)\\
&\leq \ds \max_{t>0}f_1(t)+\max_{s>0}g_1(s)\\
&<\ds B_1+\frac 1NS^{\frac N2}\ \ \ \hbox{for}~\va~\hbox{small~enough}.
\end{array}
\end{equation}
Denote
$$
\ds D_1=\int_{\R^N}|u_1|^{2^*}, \ \ D_2=\be \int_{\R^N}|u_1|^{\frac{2^*}{2}}|v_\va|^{\frac{2^*}{2}}<0,\ \ \ds D_3=\int_{\R^N}|v_\va|^{2^*},
$$
$$
D_4=\ds\int_{\R^N}(|\nabla v_\va|^2+|v_\va|^2),\ \ E_1=\la_1 \int_{\R^N}|u_1|^\al, \ \ E_2=\la_2 \int_{\R^N}|v_\va|^r.
$$
Then $$D_1+E_1=\ds\int_{\R^N}(|\nabla u_1|^2+|u_1|^2)$$
and \begin{equation}\label{3.12}
D_2^2=|\be|^2 \left(\int_{\R^N}|u_1|^{\frac{2^*}{2}}|v_\va|^{\frac{2^*}{2}}\right)^2\leq
\frac{1}{4}D_1 D_3<D_1D_3.
\end{equation}
We claim that there exist $t_\va,s_\va>0$ such that $(t_\va u_1,s_\va v_\va)\in M$, i.e. $(t_\va,s_\va)$ solves the following system
\begin{equation}\label{3.13}\left\{%
\begin{array}{ll}
    t^2(D_1+E_1)=t^{2^*}D_1+t^{\frac{2^*}{2}}s^{\frac{2^*}{2}}D_2+t^\al E_1, & \hbox{$ $} \\
    s^2D_4=t^{2^*}D_3+t^{\frac{2^*}{2}}s^{\frac{2^*}{2}}D_2+t^r E_2,& \hbox{$ $} \\
    t,s>0.& \hbox{$ $} \\
\end{array}%
\right.
\end{equation}
System \eqref{3.13} is equivalent to
\begin{equation}\label{3.14}\left\{%
\begin{array}{ll}\ds
(t^{2-\frac{2^*}{2}}-t^{\frac{2^*}{2}})D_1=s^{\frac{2^*}{2}}D_2+(t^{\al-\frac{2^*}{2}}-t^{2-\frac{2^*}{2}})E_1, & \hbox{$ $} \\
s^{2-\frac{2^*}{2}}D_4=s^{\frac{2^*}{2}}D_3+t^{\frac{2^*}{2}}D_2+s^{r-\frac{2^*}{2}}E_2,& \hbox{$ $} \\
    t,s>0,& \hbox{$ $} \\
\end{array}%
\right.
\end{equation}
then
$$\ds
(t^{2-\frac{2^*}{2}}-t^{\frac{2^*}{2}})D_1<(t^{\al-\frac{2^*}{2}}-t^{2-\frac{2^*}{2}})E_1,
$$
hence $t>1$ since $2<\al<2^*$ and $1<\frac{2^*}{2}<2$. Moreover, by \eqref{3.14}, we have that
$$
s^{\frac{2^*}{2}}=\ds\frac{(t^{2-\frac{2^*}{2}}-t^{\frac{2^*}{2}})D_1-(t^{\al-\frac{2^*}{2}}-t^{2-\frac{2^*}{2}})E_1}{D_2}.
$$
Then \eqref{3.14} is equivalent to
\begin{eqnarray*}
\ds G(t)
&\triangleq&\ds \Big[\frac{(t^{2-\frac{2^*}{2}}-t^{\frac{2^*}{2}})D_1-(t^{\al-\frac{2^*}{2}}-t^{2-\frac{2^*}{2}})E_1}{D_2}\Big]^{\frac{4-2^*}{2^*}}D_4\\
&&-
\Big[\frac{(t^{2-\frac{2^*}{2}}-t^{\frac{2^*}{2}})D_1-(t^{\al-\frac{2^*}{2}}-t^{2-\frac{2^*}{2}})E_1}{D_2}\Big]D_3-t^{\frac{2^*}{2}}D_2\\
&&-\ds\Big[\frac{(t^{2-\frac{2^*}{2}}-t^{\frac{2^*}{2}})D_1-(t^{\al-\frac{2^*}{2}}-
t^{2-\frac{2^*}{2}})E_1}{D_2}\Big]^{\frac{2r-2^*}{2^*}}E_2=0, \ \ \ \ \ \ \ t>1.
\end{eqnarray*}
Since $N>5$ and $2<r<2^*$, $2^*<4<2r$. Then $G(1)=-D_2>0$ and $\ds \lim_{t\rightarrow +\infty}\frac{G(t)}{t^{\frac{2^*}{2}}}=\frac{D_1D_3-D_2^2}{D_2}<0$, hence $G(t)=0$ has a
solution $t>1$. So \eqref{3.14} has a solution $t_\va,s_\va>0,$ i.e. \eqref{3.13} has a solution $t_\va,s_\va>0.$ Therefore $(t_\va u_1,s_\va v_\va)\in M$. By \eqref{3.11}, we have that
$$
B\leq \ds I(t_\va u_1,s_\va v_\va)\leq \max_{t,s>0}I(t u_1,s v_\va)<B_1+\frac1N S^{\frac N2}.
$$
Similarly, we can also prove that $B<B_2+\ds\frac1N S^{\frac N2}$. Moreover, by  Lemma~\ref{le2.1}(i) and \eqref{1.8},\eqref{1.9}, we see that
$$
A>\ds \max\Big\{B_1+\frac1N S^{\frac N2},\,B_2+\frac1N S^{\frac N2}\Big\}.
$$
Then we complete the proof of Lemma~\ref{le3.1}.
\end{proof}

\begin{remark}\label{re1.2} In \cite{cz1,cz2}, Chen and Zou proved the same result for any $\be<0$, however, for equation \eqref{1.1}, because of the effect of the perturbation terms, we only obtain the estimate for $\be\in[-\frac{1}{2},0)$.
\end{remark}

\begin{lemma}\label{le3.2}~~
Suppose that  $-\frac12\leq\be<0$, then there exist $C_2>C_1>0$, such that for any $(u,v)\in M$ with $I(u,v)\leq C$, then we have that
$$
C_1\leq\ds\int_{\R^N}|u|^{2^*},\,\,\int_{\R^N}|v|^{2^*}\leq C_2.
$$
\end{lemma}
\begin{proof}
Since $(u,v)\in M$ and $2<\al,r<2^*$ and $-\frac12\leq\be<0$,
\begin{eqnarray*}
\ds I(u,v)&=&\frac1N\int_{\R^N}(|u|^{2^*}+2\be |u|^{\frac{2^*}{2}}|v|^{\frac{2^*}{2}}+|v|^{2^*})+(\frac 12-\frac 1\al)\la_1\int_{\R^N}|u|^\al\\
&&+(\frac 12-\frac 1 r)\la_2\int_{\R^N}|v|^r\\
&\geq& \frac1N(1-|\be|)\int_{\R^N}(|u|^{2^*}+|v|^{2^*}),
\end{eqnarray*}
then there exists $C_2>0$ such that
$\ds\int_{\R^N}|u|^{2^*},\,\int_{\R^N}|v|^{2^*}\leq C_2.$

Now, we prove that $\ds\int_{\R^N}|u|^{2^*},\int_{\R^N}|v|^{2^*}\geq C_1$.

By $(u,v)\in M$, $\be<0$ and the Sobolev embedding inequality, we have that
\begin{equation}\label{3.15}
\begin{array}{ll}
\ds \int_{\R^N}(|\nabla u|^2+|u|^2)&=\ds\int_{\R^N}|u|^{2^*}+\be |u|^{\frac{2^*}{2}}|v|^{\frac{2^*}{2}}+\la_1\int_{\R^N}|u|^\al\\[5mm]
&\leq\ds \int_{\R^N}|u|^{2^*}+\la_1\int_{\R^N}|u|^\al\\[5mm]
&\leq\ds S^{-\frac{2^*}{2}}\Big(\int_{\R^N}(|\nabla u|^2+|u|^2)\Big)^{\frac{2^*}{2}}+C\Big(\int_{\R^N}(|\nabla u|^2+|u|^2)\Big)^{\frac{\al}{2}}.
\end{array}
\end{equation}
We conclude that
$$\int_{\R^N}(|\nabla u|^2+|u|^2)\geq C\ \ \ \hbox{for~some}~C>0$$
 since $\ds\frac{\al}{2},\frac{2^*}{2}>1$. Moreover, $I(u,v)\leq C$ implies that $(u,v)$ is bounded in $H$, i.e. $\|(u,v)\|\leq C$ for some $C>0$, then by \eqref{3.15} and the interpolation inequality, we have that
\begin{eqnarray*}
\ds C\leq\int_{\R^N}(|\nabla u|^2+|u|^2)&\leq&\int_{\R^N}|u|^{2^*}+\la_1\int_{\R^N}|u|^\al\\
&\leq&\ds \int_{\R^N}|u|^{2^*}+\la_1(\int_{\R^N}|u|^{2^*})^{\frac{\theta\alpha}{2^*}}(\int_{\R^N}|u|^2)^{\frac{(1-\theta)\alpha}{2}}\\
&\leq&\ds \int_{\R^N}|u|^{2^*}+C(\int_{\R^N}|u|^{2^*})^{\frac{\theta\alpha}{2^*}},
\end{eqnarray*}
where $\ds \frac{1}{\al}=\frac{\theta}{2^*}+\frac{1-\theta}{2}$ and $0<\theta<1$. So there exists a $C>0$ such that
$$\ds\int_{\R^N}|u|^{2^*}\geq C.$$ Similarly, we have that $\ds\int_{\R^N}|v|^{2^*}\geq C$ and we complete the proof of the Lemma.
\end{proof}

\noindent $\textbf{Proof of Theorem 1.2}$

The main idea of the proof comes from \cite{cz1,cz2}, but more careful analysis is needed. Note that $I$ is coercive and bounded from below on $M$, then by the Ekeland's variational principle (see \cite{sm}), there exists a minimizing sequence $\{(u_n,v_n)\}\subset M$ satisfying

\begin{equation}\label{3.16}
\ds I(u_n,v_n)\leq \min\{B+\frac1n,A\},
\end{equation}
\begin{equation}\label{3.17}
\ds\ I(u,v)\geq I(u_n,v_n)-\frac1n \|(u_n,v_n)-(u,v)\|,\,\,\forall (u,v)\in M.
\end{equation}
Then it is easy to see that  $\{(u_n,v_n)\}$ is bounded in $H$. For any $(\varphi,\psi)\in H$ with $\|\varphi\|_{H^1(\R^N)},\|\psi\|_{H^1(\R^N)}\leq 1$ and
each $n\in N$, define $h_n$, $g_n:\R^3\rightarrow \R$ by
\begin{equation}\label{3.18}
\begin{array}{ll}
\ds h_n(t,s,l)&=\ds \int_{\R^N}|\nabla (u_n+t\varphi+su_n)|^2+\int_{\R^N}|u_n+t\varphi+su_n|^2-\int_{\R^N}|u_n+t\varphi+su_n|^{2^*}\\
&\ \ \ -\be \ds\int_{\R^N}|u_n+t\varphi+su_n|^{\frac{2^*}{2}}|v_n+t\phi+lv_n|^{\frac{2^*}{2}}-\ds \la_1\int_{\R^N}|u_n+t\varphi+su_n|^\al
\end{array}
\end{equation}
and
\begin{equation}\label{3.19}
\begin{array}{ll}
\ds g_n(t,s,l)&=\ds \int_{\R^N}|\nabla (v_n+t\phi+lv_n)|^2+\int_{\R^N}|v_n+t\phi+lv_n|^2-\int_{\R^N}|v_n+t\phi+lv_n|^{2^*}\\
&\ \ \ -\be \ds\int_{\R^N}|u_n+t\varphi+su_n|^{\frac{2^*}{2}}|v_n+t\phi+lv_n|^{\frac{2^*}{2}}-\ds \la_2\int_{\R^N}|v_n+t\phi+lv_n|^r
\end{array}
\end{equation}

Let $\mathbf{0}=(0,0,0)$. Then $h_n,g_n\in C^1(\R^3,\R)$ and $h_n(\mathbf{0})=g_n(\mathbf{0})=0$.
$$
\frac{\partial h_n}{\partial s}(\mathbf{0})=-(2^*-2)\int_{\R^N}|u_n|^{2^*}-(\frac{2^*}{2}-2)\be\int_{\R^N}|u_n|^{\frac{2^*}{2}}|v_n|^{\frac{2^*}{2}}-(\al-2)\la_1
\int_{\R^N}|u_n|^\al,
$$
$$
\frac{\partial h_n}{\partial l}(\mathbf{0})=\frac{\partial g_n}{\partial s}(\mathbf{0})=-\frac{2^*}{2}\be\int_{\R^N}|u_n|^{\frac{2^*}{2}}|v_n|^{\frac{2^*}{2}},
$$
$$
\frac{\partial g_n}{\partial l}(\mathbf{0})=-(2^*-2)\int_{\R^N}|v_n|^{2^*}-(\frac{2^*}{2}-2)\be\int_{\R^N}|u_n|^{\frac{2^*}{2}}|v_n|^{\frac{2^*}{2}}-(r-2)\la_2
\int_{\R^N}|v_n|^r.
$$
Define the matrix
$$
F_n\triangleq
 \left (
 \begin{array}{ll}
 \ds \frac{\partial h_n}{\partial s}(\mathbf{0}) \,\,\, \frac{\partial h_n}{\partial l}(\mathbf{0})\\
 \ds \frac{\partial g_n}{\partial s}(\mathbf{0})\,\,\,\frac{\partial g_n}{\partial l}(\mathbf{0})
 \end{array}
 \right ).
$$
Then by $-\frac{1}{2}\leq\be<0$, the H\"{o}lder inequality and Lemma~\ref{le3.2}, we have that
\begin{eqnarray*}
\ds \hbox{det}(F_n)&=&\Big[(2^*-2)\int_{\R^N}|u_n|^{2^*}+\be(\frac{2^*}{2}-2)\int_{\R^N}|u_n|^{\frac{2^*}{2}}|v_n|^{\frac{2^*}{2}}+\la_1(\al-2)\int_{\R^N}|u_n|^\al\Big]\\
&&\times \Big[(2^*-2)\int_{\R^N}|v_n|^{2^*}+\be(\frac{2^*}{2}-2)\int_{\R^N}|u_n|^{\frac{2^*}{2}}|v_n|^{\frac{2^*}{2}}+\la_2(r-2)\int_{\R^N}|v_n|^r\Big]\\
&&-(\frac{2^*}{2})^2\be^2(\int_{\R^N}|u_n|^{\frac{2^*}{2}}|v_n|^{\frac{2^*}{2}})^2\\
&>&(2^*-2)^2\int_{\R^N}|u_n|^{2^*}\int_{\R^N}|v_n|^{2^*}+\be^2(\frac{2^*}{2}-2)^2(\int_{\R^N}|u_n|^{\frac{2^*}{2}}|v_n|^{\frac{2^*}{2}})^2\\
&&+\be (\frac{2^*}{2}-2)(2^*-2)\int_{\R^N}(|u_n|^{2^*}+|v_n|^{2^*})\int_{\R^N}|u_n|^{\frac{2^*}{2}}|v_n|^{\frac{2^*}{2}}\\
&&-(\frac{2^*}{2})^2\be^2(\int_{\R^N}|u_n|^{\frac{2^*}{2}}|v_n|^{\frac{2^*}{2}})^2\\
&>&(2^*-2)^2\int_{\R^N}|u_n|^{2^*}\int_{\R^N}|v_n|^{2^*}+\Big[\be^2(\frac{2^*}{2}-2)^2+2\be (\frac{2^*}{2}-2)(2^*-2)\\
&&-(\frac{2^*}{2})^2\be^2\Big]\times (\int_{\R^N}|u_n|^{\frac{2^*}{2}}|v_n|^{\frac{2^*}{2}})^2\\
&>&(2^*-2)^2\left[(2^*-2)-2\be^2+2\be(\frac{2^*}{2}-2)\right]\int_{\R^N}|u_n|^{2^*}\int_{\R^N}|v_n|^{2^*}\geq C>0,
\end{eqnarray*}
where $C$ is independent of $n$ and we have used the fact that
$$f(t)\triangleq -2t^2+(2^*-4)t+2^*-2\geq\min\{f(-\frac{1}{2}),f(0)\}>0,\ \ \forall~t\in[-\frac{1}{2},0).$$

By the Implicit Function Theorem, there exist $\delta_n>0$ and functions $s_n(t)$, $l_n(t)\in C^1(-\delta_n,\delta_n)$. Moreover, $s_n(0)=l_n(0)=0$,
$$
h_n(t,s_n(t),l_n(t))=0,\,\,g_n(t,s_n(t),l_n(t))=0,\ \ \ \forall~t\in (-\delta_n,\delta_n)
$$
and
\begin{equation*}
 \left\{
 \begin{array}{ll}
 \ds s_n'(0)=\frac{1}{\hbox{det}F_n}\Big(\frac{\partial g_n}{\partial t}(\mathbf{0})\frac{\partial h_n}{\partial l}(\mathbf{0})-\frac{\partial g_n}{\partial l}(\mathbf{0})\frac{\partial h_n}{\partial t}(\mathbf{0})\Big)\\
 l_n'(0)=\ds \frac{1}{\hbox{det}F_n}\Big(\frac{\partial g_n}{\partial s}(\mathbf{0})\frac{\partial h_n}{\partial t}(\mathbf{0})-\frac{\partial g_n}{\partial t}(\mathbf{0})\frac{\partial h_n}{\partial s}(\mathbf{0})\Big).
 \end{array}
 \right.
\end{equation*}
Since $\{(u_n,v_n)\}$ is bounded in $H$, it is easy to see that $|\frac{\partial h_n}{\partial t}(\mathbf{0})|,|\frac{\partial g_n}{\partial t}(\mathbf{0})|\leq C$, where $C$ is independence of $n$. Then by Lemma~\ref{le3.2}, we also have that
$$
|\frac{\partial h_n}{\partial s}(\mathbf{0})|,\,|\frac{\partial h_n}{\partial l}(\mathbf{0})|,\,|\frac{\partial g_n}{\partial s}(\mathbf{0})|,\,|\frac{\partial g_n}{\partial l}(\mathbf{0})|\leq C.
$$
Hence, we can conclude that
\begin{equation}\label{3.20}
|s_n'(0)|,\,\,|l_n'(0)|\leq C,
\end{equation}
where $C$ is independence of $n$.

Denote $$\varphi_{n,t}\triangleq u_n+t\varphi+s_n(t)u_n,\ \ \ \phi_{n,t}\triangleq v_n+t\phi+l_n(t)v_n,$$
 then
$(\varphi_{n,t},\phi_{n,t})\in M$ for $\forall~t\in (-\delta_n,\delta_n)$. It follows from \eqref{3.17} that
\begin{equation}\label{3.21}
 I(\varphi_{n,t},\phi_{n,t})-I(u_n,v_n)\geq -\frac1n\|(t\varphi+s_n(t)u_n,t\phi+l_n(t)v_n\|.
\end{equation}
By $(u_n,v_n)\in M$ and the Taylor Expansion we have that
\begin{equation}\label{3.22}
\begin{array}{ll}
I(\varphi_{n,t},\phi_{n,t})-I(u_n,v_n)&=\langle I'(u_n,v_n),(t\varphi+s_n(t)u_n,t\phi+l_n(t)v_n)\rangle+r(n,t)\\
&=t\langle I'(u_n,v_n),(\varphi,\phi)\rangle+r(n,t),
\end{array}
\end{equation}
where $r(n,t)=o(\|(t\varphi+s_n(t)u_n,t\phi+l_n(t)v_n)\|)$ as $t\rightarrow 0$. By \eqref{3.20}, we see that
\begin{equation}\label{3.23}
\ds\limsup_{t\rightarrow0}\|(\varphi+\frac{s_n(t)}{t}u_n, \phi+\frac{l_n(t)}{t}v_n)\|\leq C,
\end{equation}
where $C$ is independence of $n$. Hence $r(n,t)=o(t)$. By \eqref{3.21}-\eqref{3.23} and letting $t\rightarrow 0$, we have that
$$
|\langle I'(u_n,v_n)(\varphi,\phi)\rangle|\leq \frac{C}{n},
$$
where $C$ is independence of $n$. Hence
\begin{equation}\label{3.24}
\ds \lim_{n\rightarrow +\infty}I'(u_n,v_n)=0.
\end{equation}
Since $\{(u_n,v_n)\}$ is bounded in $H$, passing to a subsequence, we may assume that $(u_n,v_n)\rightharpoonup (u,v)$ in $H$ for some $(u,v)\in H$. Set $\om_n=u_n-u,\sigma_n=v_n-v$ and use the same notations as in the proof of Theorem 1.1, we also see that $I'(u,v)=0$ and \eqref{2.18}-\eqref{2.21} hold. Furthermore,
\begin{equation}\label{3.25}
0\leq I(u,v)\leq I(u,v)+\frac1N(b_1+b_2)=\ds\lim_{n\rightarrow +\infty}I(u_n,v_n)=B.
\end{equation}

\textbf{Case 1:} $u\equiv0,v\equiv0$.

By \eqref{3.25}, we have $0<NB=b_1+b_2<+\infty$. Then we conclude from Lemma 3.3 that $0<b_1<+\infty,0<b_2<+\infty$. Hence we may assume that $\om_n\neq 0,\sigma_n\neq0$ for $n$ large.
Similar to the argument used in the proof of Theorem 1.3 (P. 32-33) in \cite{cz2}, for $n$ large, there exist $t_n,s_n>0$ such that
$(t_n\omega_n,s_n\sigma_n)\in \mathcal{N}$ and $\ds\lim_{n\rightarrow +\infty}(|t_n-1|+|s_n-1|)=0$. Then we have that
$$
\frac1N(b_1+b_2)=\ds\lim_{n\rightarrow +\infty}E(\omega_n,\sigma_n)=\lim_{n\rightarrow +\infty}E(t_n\omega_n,s_n\sigma_n)\geq A.
$$
By \eqref{3.25}, we see that $B\geq A$, which is a contradiction with Lemma~\ref{le3.1}. Therefore, Case 1 is impossible.

\textbf{Case 2:} $u\not\equiv0,v\equiv0$ or $u\equiv0,v\not\equiv0$.

Without loss of generality, we may assume that $u\not\equiv 0,v\equiv0$. Then $u\in H^r_1(\R^N)$
is a nontrivial solution of
$-\Delta u+u=|u|^{2^*-2}u+\la_1|u|^{\al-2}u$, so $\ds I(u,0)\geq B_1$. By Lemma 3.3, we have that $b_2>0$. By Case 1, we may assume that $b_1=0$. Then
$$\ds\lim_{n\rightarrow +\infty}\ds\int_{\R^N}|\omega_n|^{\frac{2^*}{2}}|\sigma_n|^{\frac{2^*}{2}}=0,$$ hence
$$
\ds\int_{\R^N}|\nabla \sigma_n|^2=\int_{\R^N}|\sigma_n|^{2^*}+o(1)\leq S^{-\frac{2^*}{2}}\Big(\int_{\R^N}|\nabla \sigma_n|^2\Big)^{\frac{2^*}{2}}+o(1),
$$
which implies that $b_2\geq S^{\frac N2}$. By \eqref{3.25} we have that
$$
B\geq B_1+\frac1N b_2\geq B_1+\frac1N S^{\frac N2},
$$
which is a contradiction with Lemma~\ref{le3.1}. Therefore Case 2 is impossible.

Since Case 1 and Case 2 are both impossible, we see that $u\not\equiv 0,v\not\equiv 0$. By $I'(u,v)=0$, then we have that $(u,v)\in M$, hence by \eqref{3.25}, $I(u,v)=B$. By a similar argument as in the proof of Theorem~\ref{th1.1}, we see that $(u,v)\in H$ is a positive solution of \eqref{1.1} and $I(u,v)=B$. This completes the proof of Theorem~\ref{th1.2}.


\begin{thebibliography}{99}
\addcontentsline{toc}{section}{\protect \heiti 参考文献} 
\bibitem{at}
T. Aubin, Problems isoperimetriques et espaces de Sobolev, J. Diff. Geom., 11(1976), 573-598.

\bibitem{aa}
N. Akhmediev, A. Ankiewicz, Partially coherent solitons on a finite background, Phys. Rev. Lett., 82(1999), 2661-2664.

\bibitem{ac1}
A. Ambrosetti, E. Colorado, Bound and ground states of coupled nonlinear Schr\"{o}dinger equations, C. R. math. Acad. Sci. Paris, 342(2006),453-458.

\bibitem{ac2}
A. Ambrosetti, E. Colorado, Standing waves of some coupled nonlinear Schr\"{o}dinger equations, J. Lond. Math. Sci., 75(2007), 67-82.

\bibitem{ar}
A. Ambrosetti, P. H. Rabinowitz, Dual variational methods in critical point theory and applications, J. Func. Anal., 14(1973), 349-381.


\bibitem{bdw}
T. Bartsch, N. Dancer, Z. Q. Wang, A Liouville theorem, a priori bounds, and bifurcating branches of positive solutiond for a nonlinear elliptic system, Calc. Var. PDE., 37(2010), 345-361.

\bibitem{bn}
H. Brezis, L. Nirenberg, Positive solutions of nonlinear elliptic
problems involving critical Sobolev exponent, Comm. Pure Appl. Math.,
36(1983), 437-477.

\bibitem{bw1}
T. Bartsch, Z. Q. Wang, Note on ground states of nonlinear elliptic Schr\"{o}dinger systems, J. Partial Diff. Equ., 19(2006), 200-207.

\bibitem{bw2}
T. Bartsch, Z. Q. Wang, J. C. Wei, Bounded states for s coupled Schr\"{o}dinger system, J. Fixed point Theorey Appl., 2(2007), 353-367.

\bibitem{cz1}
Z. J. Chen, W. M. Zou, Positive least energy solutions and phase seperation for coupled Schr\"{o}dinger equations with critical exponent, Arch. Ration. Mech. Anal., 205(2012), 515-551.

\bibitem{cz2}
Z. J. Chen, W. M. Zou, Positive least energy solutions and phase separation for coupled Schr\"{o}dinger equations with critical exponent (Dim$\geq5$),  arXiv: 1209.2522v1 [math.AP].


\bibitem{dyb}
Y. B. Deng, The existence and nodal character of the solutions in $\R^n$ for semilinear elliptic equation involving critical Sobolev exponent, Acta Mathematica Scientia, 9(1989), 385-402.


\bibitem{dww}
N. Dancer, J. C. Wei, T Weth, A priori bounds versus multiple existence of positive solutiond for a nonlinear Schr\"{o}dinger systems, Ann. Inst. H. Poincar\'{e} Anal. Non Lin\'{e}aire, 27(2010), 953-969.

\bibitem{egbb}
B. Esry, C. Greene, J. Burke, J. Bohn, Hartree-Fock theory for double condensates, Phys. Rev. Lett. 78(1997), 3594-3597.


\bibitem{lw}
T. C. Lin, J. C. Wei, Ground state of $N$ coupled nonlinear Schr\"{o}inger equations in $\R^n$, $n\leq3$, Comm. Math. Phys., 255(2005), 629-653.

\bibitem{lw2}
T. C. Lin, J. C. Wei, Spikes in two coupled nonlinear Schr\"{o}dinger equations, Ann. Inst. H. Poincar\'{e} Anal. Non Lin\'{e}aire, 22(2005), 403-439.

\bibitem{lw3}
T. C. Lin, J. C. Wei, Multiple bound states of nonlinear Schr\"{o}dinger equations with trapping potentials, J. Differ. Equ., 229(2006), 743-767.

\bibitem{lw4}
Z. Liu, Z. Q. Qang, Multiple bound states of nonlinear Schr\"{o}dinger systems, Comm. Math. Phys., 282(2008), 721-731.

\bibitem{m}
C. R. Menyuk, Nonlinear pulse propagation in birefringent optical fibers, IEEE. J. Quantum Electron., 23 (1987), 174-176.


\bibitem{p}
A. Pomponio, Coupled nonlinear Schr\"{o}dinger systems with potentials, J. Differ. Equ., 227(2006), 258-281.

\bibitem{sk}
S. Kim, On vertor solutions for coupled nonlinear Schr\"{o}dinger equations with critical exponents, Comm. Pure Appl. Anal., 12(2013), 1259-1277.

\bibitem{sm}
M. Struwe, Variational methods applications to nonlinear partial differential equations and Hamiltonian systems, springer, 1996.

\bibitem{tg}
G. Talenti, Best constant in Sobolev inequality, Ann. Mat. Pure Appl., 110(1976), 352-372.

\bibitem{w1}
M. Willem,  Minimax theorems,Birkh$\ddot{a}$user, 1996.

\bibitem{ww1}
J. C. Wei, T. Weth, Nonradial symmetric bound states for a system of two coupled Schr\"{o}dinger equations, Rend. Lincei mat. Appl., 18(2007), 279-293.

\bibitem{ww2}
J. C. Wei, T. Weth, Asymptotic behavior of solutions of planar systems with strong competition, Nonlinearity, 21(2008), 305-317.

\end{thebibliography}
 \end{document}